\documentclass[10pt]{amsart}

\usepackage[T1]{fontenc}
\usepackage[utf8]{inputenc}
\usepackage{amsmath,amssymb,amsthm,mathtools}
\usepackage[margin=1in]{geometry}

\usepackage{booktabs,array,longtable,microtype,enumitem}
\usepackage{listings}
\usepackage[hidelinks]{hyperref}
\numberwithin{equation}{section}

\theoremstyle{definition}

\newtheorem{theorem}{Theorem}[section]
\newtheorem{lemma}[theorem]{Lemma}
\newtheorem{proposition}[theorem]{Proposition}
\newtheorem{corollary}[theorem]{Corollary}

\newtheorem{remark}[theorem]{Remark}
\newtheorem{conj}[theorem]{Conjecture}

\newcommand{\PP}{\mathbb P}
\newcommand{\vir}{\mathrm{vir}}
\newcommand{\CC}{\mathbb C}
\newcommand{\ZZ}{\mathbb Z}
\newcommand{\QQ}{\mathbb Q}
\newcommand{\OO}{\mathcal O}

\newcommand{\cA}{\mathcal A}
\newcommand{\Ext}{\operatorname{Ext}}

\newcommand{\ch}{\operatorname{ch}}
\newcommand{\rk}{\operatorname{rk}}
\newcommand{\supp}{\operatorname{supp}}
\newcommand{\Tot}{\operatorname{Tot}}
\newcommand{\DT}{\mathrm{DT}}

\newcommand{\Norm}[1]{\left\lVert #1\right\rVert^2}

\title[A proof of GW/$\mathrm{DT}_4$ conjecture on local projective plane]{A proof of GW/$\textbf{DT}_4$ conjecture \\ on local projective plane}
\author{Yalong Cao}
\address{Morningside Center of Mathematics, Institute of Mathematics \& State Key Laboratory of Mathematical Sciences, Academy of Mathematics and Systems Sciences, Chinese Academy of Sciences, Beijing, China}
\email{yalongcao@amss.ac.cn}


\begin{document}
\begin{abstract}
We prove a closed formula for $\mathrm{DT}_4$ invariants of one dimensional stable sheaves on the total space of
$\mathcal{O}_{\mathbb{P}^2}(-1)\oplus\mathcal{O}_{\mathbb{P}^2}(-2)$ for all degree curve classes. During the proof, orientations 
of moduli spaces are chosen in a canonical way. 
Combining with Gopakumar-Vafa type invariants proposed and computed by
Klemm--Pandharipande, we verify a conjecture proposed by the author, Maulik and Toda. 
\end{abstract}
\maketitle


\subsection{Background}
On Calabi--Yau 3-folds, enumerative geometry of curves has been extensively studied from many different perspectives in the past thirty years. 
In 2007, Klemm--Pandharipande \cite{KP} proposed that it is equally interesting to explore a similar story on Calabi-Yau 4-folds. 
They defined \textit{Gopakumar--Vafa} type BPS invariants from Gromov-Witten invariants of compact Calabi-Yau 4-folds $X$ and conjectured their integrality.
Such invariants are nontrivial only in genus $0$ and $1$ by dimensional reason, see \S \ref{sect on GWGV} for details.  

The author, Maulik and Toda \cite{CMT1}
proposed a \textit{sheaf-theoretic} interpretation of the genus zero
Gopakumar--Vafa type invariants $n_{0,\beta}(\gamma)$ of Klemm--Pandharipande \cite{KP}\footnote{There are other sheaf-theoretic interpretations by other moduli spaces, see e.g.~\cite{CMT2, CT1}}.
For a curve class $\beta$ and an insertion $\gamma\in H^4(X)$, the
conjectural correspondence takes the form
\begin{equation}\label{eq:CMT0}n_{0,\beta}(\gamma)=\langle\tau_0(\gamma) \rangle_{\beta}^{\DT_4}, \end{equation}
or equivalently 
\begin{equation}\mathrm{GW}_{0,\beta}(\gamma)
=\sum_{k\geqslant1, k\mid\beta}\frac{1}{k^2}\langle\tau_0(\gamma) \rangle_{\beta/k}^{\DT_4},
\label{eq:CMT}
\end{equation}
where the LHS is the genus zero Gromov-Witten invariant of $X$ with primary insertion $\gamma$, and 
$\langle\tau_0(\gamma) \rangle_{\beta}^{\DT_4}$ is the \textit{primary} $\DT_4$ invariant for one dimensional stable sheaves with curve class $\beta$ and Euler characteristic one. 
Note that $\DT_4$ invariants depend on the choice of \textit{orientation}, see \S \ref{sect on GVDT4} for more details. 

To capture the genus one Gopakumar--Vafa type invariants $n_{1,\beta}$ of \cite{KP}, \textit{descendent} $\DT_4$ invariants $\langle\tau_1(\alpha) \rangle_{\beta}^{\DT_4}$ are required. 
It was conjectured in \cite{CT2} that
there is an equality of functions in $\alpha \in H^2(X)$:
\begin{align}\label{eq:CT2}\langle\tau_1(\alpha) \rangle_{\beta}^{\DT_4}=\frac{n_{0,\beta}(\alpha^2)}{2\,(\alpha\cdot\beta)}
-\sum_{\beta_1+\beta_2=\beta}\frac{(\alpha\cdot\beta_1)(\alpha\cdot\beta_2)}{4\,(\alpha\cdot\beta)}m_{\beta_1,\beta_2}
-\sum_{k\geqslant1,\,k|\beta}\frac{(\alpha\cdot\beta)}{k}\,n_{1,\beta/k},  \end{align}
where $m_{\beta_1,\beta_2}$ are meeting invariants, 
which can be inductively deduced from genus 0 invariants $n_{0,\beta}(-)$. 
Combining \eqref{eq:CMT0}, \eqref{eq:CT2} and M\"obius inversion formula, we obtain a sheaf theoretic interpretation of genus one Gopakumar--Vafa type invariants $n_{1,\beta}$ in terms of  
$\DT_4$ invariants. 

In this paper, we consider the (4-fold version of) local projective plane: 
\begin{equation}
\pi:X=\Tot_{\PP^2}\bigl(\OO(-1)\oplus\OO(-2)\bigr)\rightarrow\PP^2.
\label{eq:X}
\end{equation}
The relevant moduli spaces of stable maps and stable sheaves are proper, although $X$ is not, so the conjectural correspondence still makes sense.
One advantage of working with Calabi-Yau 4-fold of type $X=\Tot(K_Y)$, the total space of canonical bundle of a Fano 3-fold\,\footnote{Here we do not require $Y$ to be proper.} $Y$, is that moduli stacks of coherent sheaves have \textit{canonical choice} of orientations, e.g. \cite[Thm.~4.3]{CZZ}. Therefore we can get rid of difficulty caused by both existence and choice of orientations \cite{CGJ1, CGJ2, JU, K}. This makes finding exact closed formula of invariants possible. 

Moreover, for moduli spaces of one dimensional stable sheaves, the correct choice of orientation which makes \eqref{eq:CMT} work 
is conjectured in \cite[Eqn.\,(0.1)]{Cao}. That is given by the canonical choice of orientation from presenting $X$ as $\Tot(K_Y)$ with correction made by a
sign depending on the pairing of the first Chern class of $Y$ and the curve class. 

\subsection{Main result}
There are computations in small degree curve classes for different geometries supporting conjectures \eqref{eq:CMT}, \eqref{eq:CT2}, e.g.~\cite{CMT1, Cao1, CLW, CW}.
The main result of this paper explicitly computes both $\langle\tau_0(\gamma) \rangle_{d}^{\DT_4}$, $\langle\tau_1(\alpha) \rangle_{d}^{\DT_4}$
for local $\PP^2$ \eqref{eq:X} in all degree curve classes, and verifies \eqref{eq:CMT} in all degree curve classes, and \eqref{eq:CT2} when $d\leqslant 100$.
\begin{theorem}(Theorem \ref{thm on g=0})
With orientation proposed by \cite{Cao}, for any $d\geqslant 1$, we have 
$$\big\langle\tau_0(\pi^*H^2) \big\rangle_{d}^{\DT_4}=\frac{1}{2d^2}\sum_{k\geqslant1, k\mid d}
\mu(k)(-1)^{d/k}\binom{2d/k}{d/k},$$
where $H$ is the hyperplane class on $\PP^2$, $\mu$ is the M\"obius function.
In particular, \eqref{eq:CMT} holds.
\end{theorem}

\begin{theorem}(Theorem \ref{thm on g=1}, Corollary \ref{cor on d100})
With orientation proposed by \cite{Cao}, for any $d\geqslant 1$, we have 
$$\langle\tau_1(\pi^*H) \rangle_{d}^{\DT_4}=(-1)^{d+1}\sum_{a\in \cA_d}\left(
2\sum_{r=1}^{d-1}\sum_{i=0}^{r-1}(a_i+1)-\frac12\sum_{i=0}^{d-1} (a_i+1)^2-(d-1)
\right),
$$
where $\cA_d$ is a finite set given by 
$$\cA_d=\left\{a=(a_0,\ldots,a_{d-1})\in\ZZ^d\, \bigg|
\begin{array}{l}
\sum_{i=0}^{d-1} a_i=1-d,\quad a_{i+1}\geqslant a_i-1\ (0\leqslant  i<d-1),\\
\sum_{i=j}^{d-1}(a_i+1)\leqslant 0\quad(1\leqslant  j<d)
\end{array}\right\}.
$$
In particular, we verify \eqref{eq:CT2} for $d\leqslant 100$.
\end{theorem}
The verification of \eqref{eq:CT2} is through 
comparison with genus one GV type invariants (that are computed via the recursion formula \eqref{d:eq:genus1} and \cite[\S 3.2]{KP}), 
which we are only able to do for finite degrees at present.

\subsection{Outline of the proof}
By a stability argument, one dimensional stable sheaves on \eqref{eq:X} are scheme theoretically supported on the zero section. 
From that we know the moduli space is smooth, and the $\DT_4$ virtual class is the Euler class of a vector bundle $V_d$ \eqref{equ on Md vir class}.

To compute the integral, we use torus localization to reduce it to the torus fixed loci. However, torus fixed loci for one dimensional stable 
sheaves are in general \textit{not} isolated and not easy to work with combinatorically, compared to ideal sheaves or stable pairs, see e.g.~\cite{CM, Ko}.  

The key argument we employ is a vanishing result saying that only components whose elements $F$ are supported only on the line\footnote{Here we linearize $\OO(-1)$ as $\OO(-L_0)$.} $L_1=\{x_1=0\}$ or $L_2=\{x_2=0\}$ in $\PP^2$ 
contribute nontrivially to the invariant (Proposition \ref{prop:vanishing}). Here the Calabi-Yau 4-fold geometry plays a key role where we use an identity relating the rank of the torus fixed subbundle of $V_d$ and dimension 
of the component (Lemma \ref{lem:boundary}). 

Importantly, the remaining components 
are finitely many reduced points (Proposition \ref{prop:classification}), which are in one to one correspondence to the orbits of
break divisors in graphs (Proposition \ref{prop:bijection}). Then the $\DT_4$ invariant equals to an (overall) sign count 
of that finite set (Proposition \ref{prop:localization}), which we use the orbit counting theorem of Konvalinka-Reineke-Tewari \cite[Thm.~3.7]{KRT} to explicitly determine (Theorem \ref{thm:KRT}). 
Similarly for the case of descendent $\DT_4$ invariants, see \S \ref{sect on proof of thm g=1}. 

Because of the connection to the work of \cite{KRT}, which are related to DT invariants of 2-loop quivers, we give an explanation to the phenomenon that 
$\DT_4$ invariants interact with quiver descriptions as observed by \cite[\S 1.4.3]{BBvG} through torus fixed point analysis.

\subsection{Related geometry}
There are related Calabi-Yau 4-folds where one can apply similar analysis to calculate their $\DT_4$ invariants.  
One is given by the local $\PP^1\times \PP^1$:
$$\Tot_{\PP^1\times \PP^1}\bigl(\OO(-1,-1)\oplus\OO(-1,-1)\bigr), $$
where one can show \eqref{eq:CMT} for all degree curve classes and verify \eqref{eq:CT2} in finite degrees, 
by comparing the computations with \cite[\S 3.3]{KP}.

Another is the cotangent bundle $T^*\PP^2$ of $\PP^2$, where the existence of symplectic structure forces the ordinary invariants vanish. One can compute its (nontrivial) reduced $\DT_4$
invariants and verify a reduced version of \eqref{eq:CMT} (see \cite[Conj.\,0.5 (i), \S 6.2]{COT} for the formulation, and computations in small degrees).

We leave the details to a future writing, or an interested reader to check.

\subsection*{Acknowledgments}
The author is grateful to Davesh Maulik and Yukinobu Toda for the related discussions and collaborations, 
without which the GW/$\DT_4$ conjecture would not be known to the author.
He also thanks Martijn Kool for many years of discussions on torus localization and vertex formalism in DT theory.  

\subsection*{Statements and Declarations}
We have no conflicts of interest to disclose.

\subsection*{Declaration on the use of AI} The author used ChatGPT-6 Astra to search references, compute examples and discuss proof strategy in this work. 
The paper is written by the authors, who verified all  arguments and references, and take full responsibility for the content.

\section{GW/GV/$\DT_4$ correspondences}

\subsection{GW/GV correspondence}\label{sect on GWGV}

Let $X$ be a smooth projective Calabi-Yau 4-fold.
The genus 0 Gromov-Witten invariants on $X$ are defined using
insertions: for $\gamma \in H^{4}(X, \mathbb{Z})$,
one defines
\begin{equation}
\mathrm{GW}_{0, \beta}(\gamma)
=\int_{[\overline{M}_{0, 1}(X, \beta)]^{\rm{vir}}}\mathrm{ev}^{\ast}(\gamma)\in\mathbb{Q},
\nonumber \end{equation}
where $\mathrm{ev} \colon \overline{M}_{0, 1}(X, \beta)\to X$
is the evaluation map.

The \textit{genus 0 Gopakumar-Vafa type invariants} 
\begin{align}\label{g=0 GV}
n_{0, \beta}(\gamma) \in \mathbb{Q}
\end{align}
are defined by Klemm-Pandharipande \cite{KP} from the identity
\begin{align*}
\sum_{\beta>0}\mathrm{GW}_{0, \beta}(\gamma)q^{\beta}=
\sum_{\beta>0}n_{0, \beta}(\gamma) \sum_{d=1}^{\infty}
d^{-2}q^{d\beta}.
\end{align*}
For the genus 1 case, virtual dimensions of moduli spaces of stable maps are zero, so Gromov-Witten invariants
\begin{align*}
\mathrm{GW}_{1, \beta}=\int_{[\overline{M}_{1, 0}(X, \beta)]^{\rm{vir}}}
1 \in \mathbb{Q}
\end{align*}
can be defined
without insertions.
The \textit{genus 1 Gopakumar-Vafa type invariants}
\begin{align}\label{g=1 GV}
n_{1, \beta} \in \mathbb{Q}
\end{align}
 are defined in~\cite{KP} by the identity
\begin{align}\label{equ on GW1}
\sum_{\beta>0}
\mathrm{GW}_{1, \beta}q^{\beta}=
&\sum_{\beta>0} n_{1, \beta} \sum_{d=1}^{\infty}
\frac{\sigma(d)}{d}q^{d\beta}
+\frac{1}{24}\sum_{\beta>0} n_{0, \beta}(c_2(X))\log(1-q^{\beta}) \\ \nonumber
&-\frac{1}{24}\sum_{\beta_1, \beta_2}m_{\beta_1, \beta_2}
\log(1-q^{\beta_1+\beta_2}),
\end{align}
where $\sigma(d)=\sum_{i|d}i$ and $m_{\beta_1, \beta_2}\in\mathbb{Z}$ are called meeting invariants defined as follows.

Take a basis $S_1, \cdots, S_k$ of the free part of $H^4(X, \mathbb{Z})$ and let
\begin{align}\label{kunneth decom}
\sum_{i, j} g^{ij} [S_i \otimes S_j] \in H^8(X \times X, \mathbb{Z})
\end{align}
be the K\"unneth
decomposition of the $(4,4)$ component of the diagonal class (mod torsion).
For $\beta_1, \beta_2 \in H_2(X, \mathbb{Z})$,
the \emph{meeting invariant} 
\begin{align}\label{meeting invs}
m_{\beta_1, \beta_2} \in \mathbb{Z}
\end{align}
is the virtual number of rational curves
of class $\beta_1$ meeting rational curves of
class $\beta_2$. They are uniquely determined by
the following rules: \\

(i) The meeting invariants are symmetric,
$m_{\beta_1, \beta_2}=m_{\beta_2, \beta_1}$. \\

(ii) If either $\deg(\beta_1)\leqslant 0$ or $\deg(\beta_2)\leqslant 0$,
we have $m_{\beta_1, \beta_2}=0$. \\

(iii)
If $\beta_1 \neq \beta_2$, then
\begin{align}\label{meeting inv recursion 1}
m_{\beta_1, \beta_2}=
\sum_{i, j} n_{0, \beta_1}(S_i) g^{ij} n_{0, \beta_2}(S_j)
+m_{\beta_1, \beta_2-\beta_1}+m_{\beta_1-\beta_2, \beta_2}.
\end{align}

(iv)
If $\beta_1=\beta_2=\beta$, we have
\begin{align}\label{meeting inv recursion 2}
m_{\beta, \beta}=n_{0, \beta}(c_2(X))
+\sum_{i, j}n_{0, \beta}(S_i)g^{ij} n_{0, \beta}(S_j)
-\sum_{\beta_1+\beta_2=\beta}
m_{\beta_1, \beta_2}.
\end{align}
In~\cite{KP}, both of the invariants (\ref{g=0 GV}), (\ref{g=1 GV})
are conjectured to be integers, and verified in many examples.

For $g\geqslant2$, GW invariants vanish for dimensional reasons, so a GW/GV type integrality conjecture on compact Calabi-Yau 4-folds makes sense only for genus 0 and 1.

\subsection{GV/$\mathbf{DT_4}$ correspondence}\label{sect on GVDT4}

On the sheaf side, we consider the moduli 
space $M_{\beta}$ of one dimensional stable sheaves $F$ with $[F]=\beta$ and $\chi(F)=1$.

There is a virtual class constructed by \cite{CL} in special cases, and by \cite{BJ, OT} in general: 
$$[M_{\beta}]^{\vir}\in H_2(M_{\beta},\mathbb{Z})\,,$$
provided that an orientation is constructed, see \cite{CGJ1, CGJ2, JU}.
Let
\begin{align}\label{norm univ sheaf}\mathbb{F}_{\mathrm{norm}}:=\mathbb{F}\otimes\pi_M^*(\det(\mathbb{R}\pi_{M*}\mathbb{F}))^{-1} \end{align}
be the normalized universal sheaf, where $\mathbb{F}\to M_{\beta}\times X$ is any universal sheaf.
It is independent of the choice of $\mathbb{F}$ and its push-forward via $\pi_M$ to $M_\beta$ has trivial determinant. 
Set 
\begin{equation}\label{tau gamma}
\tau_0 \colon H^{4}(X, \mathbb{Z})\to H^{2}(M_\beta, \mathbb{Q}), \quad 
\tau_0(\gamma):=(\pi_{M})_{\ast}(\pi_X^{\ast}\gamma \cup\ch_{3}(\mathbb{F}_{\mathrm{norm}})),
\end{equation}
\begin{align*}  
\tau_1 \colon H^{2}(X, \mathbb{Z})\to H^{2}(M_\beta, \mathbb{Q}), \quad 
\tau_1(\alpha):=(\pi_{M})_{\ast}(\pi_X^{\ast}\alpha \cup\ch_{4}(\mathbb{F}_{\mathrm{norm}}) ),
\end{align*}
and define $\DT_4$ invariants with primary insertions:
\begin{align}\label{prim invs}  
\langle\tau_0(\gamma) \rangle_{\beta}^{\DT_4}:=\int_{[M_{\beta}]^{\rm{vir}}}\tau_0(\gamma), \,\,\, \mathrm{where}\,\,\gamma\in H^4(X,\mathbb{Z}), \end{align}
and $\DT_4$ invariants with descendent insertions:
\begin{align}\label{des invs}  
\langle\tau_1(\alpha) \rangle_{\beta}^{\DT_4}:=\int_{[M_{\beta}]^{\rm{vir}}}\tau_1(\alpha), \,\,\, \mathrm{where}\,\,\alpha\in H^2(X,\mathbb{Z}). \end{align}
In the definition of primary $\DT_4$ invariants, one may replace $\mathbb{F}_{\mathrm{norm}}$ by any universal sheaf $\mathbb{F}$ as $\ch_3(\mathbb{F})$ is dual to the fundamental 
class of $\mathbb{F}$.  
\begin{conj}$($\cite[Conj.~0.2]{CMT1}, \cite[Conj.~1.2]{CT2}$)$
\label{main conj}
For a  choice of orientation, the following holds:   
\begin{enumerate}
\item for any $\gamma\in H^4(X)$,
$$\langle\tau_0(\gamma) \rangle_{\beta}^{\DT_4}=n_{0,\beta}(\gamma), $$
\item
there is an equality of functions of $\alpha \in H^2(X)$:
\begin{align*}\langle\tau_1(\alpha) \rangle_{\beta}^{\DT_4}=\frac{n_{0,\beta}(\alpha^2)}{2\,(\alpha\cdot\beta)}
-\sum_{\beta_1+\beta_2=\beta}\frac{(\alpha\cdot\beta_1)(\alpha\cdot\beta_2)}{4\,(\alpha\cdot\beta)}m_{\beta_1,\beta_2}
-\sum_{k\geqslant1,\,k|\beta}\frac{(\alpha\cdot\beta)}{k}\,n_{1,\beta/k},  \end{align*}
where $n_{0,\beta}(-)$ and $n_{1,\beta}$ are genus 0 and 1 GV type invariants \eqref{g=0 GV}, \eqref{g=1 GV} and $m_{\beta_1,\beta_2}$ are meeting invariants \eqref{meeting invs} 
which can be inductively determined by genus 0 GV type invariants,
\item when $X$ is the total space of  canonical bundle of a Fano 3-fold, an explicit choice of orientation in \cite[Eqn.~(0.1)]{Cao} makes the above two equations hold.  
\end{enumerate}
\end{conj}


\subsection{Main results}

In this paper, we focus on the case when 
\begin{equation}
\pi: X=\Tot_{\PP^2}\bigl(\OO(-1)\oplus\OO(-2)\bigr)\longrightarrow\PP^2.
\label{eq:X}
\end{equation}
Let $H=c_1(\OO_{\PP^2}(1))$, $\gamma=\pi^*H^2$, and label curve class $\beta$ by $d\in \mathbb{Z}_{\geqslant 0}$. 
Although $X$ is not projective, the moduli stacks of stable maps and stable sheaves are proper, so the above definitions and conjectures still make sense. 

As noted in \cite[Prop.~3.1]{CMT1}, $M_d$ is isomorphic to the similar moduli space of stable sheaves on $\PP^2$:
$$\pi_*: M_d=M_d(X)\cong M_d(\PP^2)\,,$$
which is smooth and its $\DT_4$ virtual class is the Euler class of the bundle 
$$V_d:=\mathcal{E}xt^1_{\pi_{M}}(\mathbb{F},\mathbb{F}\boxtimes \OO_{\mathbb{P}^2}(-1)), $$
where $\mathbb{F}$ is a universal sheaf on $M_{d}(\PP^2)\times \PP^2$. With the choice of orientation by \cite[Eqn.~(0.1)]{Cao}, 
\begin{equation}\label{equ on Md vir class}[M_d]^{\vir}=-e(V_d)\cap [M_d], \end{equation}
because $(-1)^{c_1(Y)\cdot dH-1}=(-1)^{2d-1}=-1$ when $Y=\Tot_{\PP^2}\OO(-1)$.

For primary insertion, we denote the cycle map 
$$\sigma: M_d\to|\OO_{\PP^2}(d)|,$$
with the pullback hyperplane class 
$$h=\sigma^*c_1\bigl(\OO_{|\OO_{\PP^2}(d)|}(1)\bigr) $$
\begin{lemma}
With the above choice of orientation, we have 
$$\langle\tau_0(\pi^*H^2) \rangle_{d}^{\DT_4}=-\int_{M_d}h\cdot e(V_d), $$
$$\langle\tau_1(\pi^*H) \rangle_{d}^{\DT_4}=-\int_{M_d}\tau_1(\pi^*H)\cdot e(V_d),$$
where 
\begin{equation}\label{equ on tau1}
\tau_1(\pi^*H)=\pi_{M*}\left(H\ch_2(\mathbb{F})+\frac32H^2\ch_1(\mathbb{F})\right)-d\,c_1(\det \mathbb{R}\pi_{M*}\mathbb{F}), \end{equation}
and $\mathbb{F}$ is any universal sheaf. 
\end{lemma}
\begin{proof}
This is a straightforward calculation and we omit the details. 
\end{proof}
\begin{theorem}\label{thm on g=0}
Notations as above, we have 
$$\langle\tau_0(\pi^*H^2) \rangle_{d}^{\DT_4}=\frac{1}{2d^2}\sum_{k\geqslant1, k\mid d}
\mu(k)(-1)^{d/k}\binom{2d/k}{d/k},$$
where $\mu$ is the M\"obius function that satisfies $\sum_{k|d}\mu(k)=\delta_{d,1}$. 

In particular, Conjecture \ref{main conj}\,(1) holds for local $\PP^2$ \eqref{eq:X} with orientation specified in Conjecture \ref{main conj}\,(3) and \eqref{equ on Md vir class}.
\end{theorem}

\begin{theorem}\label{thm on g=1}
Notations as above, we have 
$$\langle\tau_1(\pi^*H) \rangle_{d}^{\DT_4}=(-1)^{d+1}\sum_{a\in \cA_d}\left(
2\sum_{r=1}^{d-1}\sum_{i=0}^{r-1}(a_i+1)-\frac12\sum_{i=0}^{d-1} (a_i+1)^2-(d-1)
\right),
$$
where $$\cA_d=\left\{a=(a_0,\ldots,a_{d-1})\in\ZZ^d \bigg|
\begin{array}{l}
\sum_{i=0}^{d-1} a_i=1-d,\quad a_{i+1}\geqslant a_i-1\ (0\leqslant  i<d-1),\\
\sum_{i=j}^{d-1}(a_i+1)\leqslant 0\quad(1\leqslant  j<d)
\end{array}\right\}.
$$
\end{theorem}
The above theorem completely computes descendent $\DT_4$ invariants 
of local $\PP^2$. One can compare them with Gromov-Witten invariants computed in \cite[\S 3.2]{KP}. 
Although we have not found a proof of Conjecture \ref{main conj}\,(2) for all degree curve classes, the following gives a computational check for 
$d\leqslant 100$. 
\begin{corollary}\label{cor on d100}
Conjecture \ref{main conj}\,(2), (3) hold with orientation specified by \eqref{equ on Md vir class} for local $\PP^2$ 
when $d\leqslant 100$. 
\end{corollary}
\begin{proof}
To compute meeting invariants $m_{i,j}$, we use the compactification
and intersection matrix in \cite[Prop.~2.11]{CT2}.
Then the recursion relations \eqref{meeting inv recursion 1}, \eqref{meeting inv recursion 2} become
\begin{align}\label{equ on meet inv}
m_{i,j}&=0\,, \quad \text{if }i\leqslant 0\text{ or }j\leqslant 0,  \\   \nonumber
m_{i,j}&=2\,n_{0,i}(\pi^*H^2)\,n_{0,j}(\pi^*H^2)+m_{i,j-i}+m_{i-j,j}\,, \quad \text{if }i,j>0,\ i\ne j,\label{d:eq:meeting}\\ \nonumber 
m_{i,i}&=-4\,n_{0,i}(\pi^*H^2)+2\,(n_{0,i}(\pi^*H^2))^2-\sum_{r=1}^{i-1}m_{r,i-r}.  \nonumber
\end{align}
By \cite[\S 3.2]{KP}, we have 
\begin{equation}\label{equ on n0dGW1}
n_{0,d}(\pi^*H^2)=\frac{1}{2d^2}\sum_{k\geqslant1, k\mid d}
\mu(k)(-1)^{d/k}\binom{2d/k}{d/k}, \quad \mathrm{GW}_{1, d}=\frac{(-1)^d4^d}{24d}.  
\end{equation}
Denote $\sigma(k)=\sum_{r\mid k}r$, \eqref{equ on GW1} in this case becomes 
\begin{equation}
\sum_{d\geqslant 1}\mathrm{GW}_{1, d}q^d
=\sum_{d\geqslant 1}n_{1,d}\sum_{k\geqslant 1}\frac{\sigma(k)}k q^{kd}
 -\frac1{24}\sum_{d\geqslant1}(4n_{0,d}(\pi^*H^2)+M_d)\log(1-q^d), \nonumber
\end{equation}
where 
\begin{equation}\label{equ on Md}
M_d=\sum_{i=1}^{d-1}m_{i,d-i}. \end{equation}
By reversing the formula, we get the recursion formula of genus one GV invariants:
\begin{equation}
n_{1,d}=\frac{(-1)^d4^d}{24d}
-\sum_{\substack{k\mid d\\ k>1}}\frac{\sigma(k)}{k}n_{1,d/k}
-\sum_{\substack{k\mid d\\ k\geqslant 1}}\left(\frac{n_{0,d/k}(\pi^*H^2)}{6k}
+\frac{M_{d/k}}{24k}\right),
\label{d:eq:genus1}
\end{equation}
where $n_{0,d}(\pi^*H^2)$ is given by \eqref{equ on n0dGW1}, $M_d$ is the sum of meeting invariants \eqref{equ on Md}, and the meeting invariants are recursively determined by 
\eqref{equ on meet inv}.
One then uses a computer program to check the claim for $d\leqslant 100$.
\end{proof}

\section{Proof of Theorem \ref{thm on g=0}}

\subsection{A vanishing result at torus fixed components}
\label{sec:vanishing}

Let $T=(\CC^*)^2$ act on $\PP^2$ by
\begin{equation}\label{equ on torus action}
(u,v)\cdot[x_0:x_1:x_2]=[x_0:u^{-1}x_1:v^{-1}x_2].
\end{equation}
Let $L_i=\{x_i=0\}\cong \PP^1$, and $p_0,p_1,p_2$ be the $T$-fixed points. 
We linearize $\OO(-1)$
so that
\begin{equation}
0\rightarrow\OO(-1)\xrightarrow{x_0}\OO
\rightarrow\OO_{L_0}\rightarrow0
\label{eq:line-sequence}
\end{equation}
is equivariant. 
Torus localization on the smooth projective variety $M_d=M_d(X)\cong M_d(\PP^2)$ gives
\begin{equation}
\int_{M_d}h\cdot e(V_d)=\sum_{B\subset M_d^T}
\int_B\frac{h_T|_B\,e_T(V_d|_B)}{e_T(N_{B/M_d})},
\label{eq:localization}
\end{equation}
where the sum is over torus fixed connected components, 
$h_T$ is the natural lift of the hyperplane class $h$.
The equivariant
pushforward lies in $H_T^0(\mathrm{pt})=\QQ$ and recovers the ordinary
integral after forgetting equivariant parameters. In particular, its value does
not depend on the chosen lifts. 

A stable torus-fixed sheaf $[F]\in M_d^T$ admits an equivariant structure, e.g.~\cite{Ko}, unique up to an overall character.
We define the $K$-theory class of the derived pullback of $F$ to a point $\iota_{p_i}:p_i\to \PP^2$: 
\[
P_i(F):=[\mathbb{L}\iota_{p_i}^*F]\in\ZZ[u^{\pm1},v^{\pm1}],\quad i=1,2.
\]
For a $K$-theory class $E=\sum_i c_{i}\alpha_i$, where the characters $\alpha_i$ are
distinct, we define a `norm':
$$\Norm E=\sum_i c_i^2. $$
If $[F]\in B\subset M_d^T$, write
\[
\mathrm{def}(F):=\dim B=\dim\Ext^1(F,F)^T, \quad \mathrm{obs}(F):=\rk(V_d|_B)^T.
\]

\begin{lemma}
\label{lem:boundary}
We have
\begin{equation}
1-\mathrm{def}(F)+\mathrm{obs}(F)=\frac{\Norm{P_1(F)}+\Norm{P_2(F)}}{2}.
\label{eq:norm}
\end{equation}
\end{lemma}
\begin{proof}
Set $G=\mathbb{L}\iota_{L_0}^*F$, the sequence
\eqref{eq:line-sequence} and  projection formula imply
\begin{equation}\label{QchiFFGG}
Q:=\chi_T(F,F)-\chi_T(F,F(-1))=\chi_T(G,G).
\end{equation}
The cotangent characters of $L_0$ at $p_1,p_2$ are $v/u$ and $u/v$.
Localization on $L_0\cong\PP^1$ yields
\[
Q=\frac{P_1(F)^\vee P_1(F)}{1-v/u}+\frac{P_2(F)^\vee P_2(F)}{1-u/v}.
\]
Since numerators are self-dual, and 
$(1-z)^{-1}+(1-z^{-1})^{-1}=1$, we obtain
\[
Q+Q^\vee=P_1(F)^\vee P_1(F)+P_2(F)^\vee P_2(F).
\]
After cancellation this is an identity of finite Laurent polynomials.
Taking trivial character terms and using the equivariant $K$-theory expression of the tangent and obstruction spaces:
\begin{equation}
[T_FM_d]=1-\chi_T(F,F),\quad
[V_{d}|_{F}]=-\chi_T(F,F(-1)),
\label{eq:ext}
\end{equation}
we get 
\begin{equation*}2(1-[T_FM_d]+[V_{d}|_{F}])^{T\emph{-}\mathrm{fixed}}= 2(\chi_T(G,G))^{T\emph{-}\mathrm{fixed}}=\Norm{P_1(F)}+\Norm{P_2(F)}. \qedhere \end{equation*}
\end{proof}

\begin{lemma}
\label{lem:support}
For $i=1,2$, we have $P_i(F)=0$ if and only if $p_i\notin\supp(F)$.
If $P_i(F)\ne0$, then $\Norm{P_i(F)}\geqslant 2$.
\end{lemma}
\begin{proof}
The first assertion can be easily checked locally. 
Since $F$ has rank zero, $P_i(F)|_{u=v=1}=0$. A nonzero integral
Laurent polynomial whose coefficients sum to zero has squared
coefficient norm at least two.
\end{proof}

\begin{proposition}
\label{prop:vanishing}
A connected component $B\subset M_d^T$ contributes zero to
\eqref{eq:localization} unless sheaves in it satisfy
\begin{equation}
\supp(F)=L_1\text{ or }L_2,\,\,\, \mathrm{and}\,\,\, 
\Norm{P_1(F)}+\Norm{P_2(F)}=2.
\label{eq:survivors}
\end{equation}
Here support is understood set-theoretically.
\end{proposition}
\begin{proof}
If both $P_1(F),P_2(F)$ are nonzero, Lemmata~\ref{lem:boundary}
and~\ref{lem:support} give $\mathrm{obs}(F)>\mathrm{def}(F)$. The same inequality holds if one
$P_i(F)$ is zero and the other has squared norm greater than two.
Decompose $V_d|_B$ into its fixed and movable subbundles. Its equivariant
Euler class contains the factor
\[
c_{\mathrm{obs}(F)}\bigl((V_d|_B)^T\bigr)=0,
\]
which vanishes since $\mathrm{obs}(F)>\dim B$. Hence the contribution of the whole component vanishes.

Every invariant curve in $\PP^2$ is supported on coordinate lines.
An invariant curve avoiding $p_1$ is supported on $L_1$, and one
avoiding $p_2$ is supported on $L_2$. A nonempty such curve cannot avoid
both points. Lemma~\ref{lem:support} then gives \eqref{eq:survivors}.
\end{proof}

\subsection{The remaining fixed sheaves}
We consider sheaves $F$ satisfying \eqref{eq:survivors} and supported
on $L_1$ (those that supported on $L_2$ can be treated similarly). Projection away from $p_1$ to $L_1$ gives an isomorphism
\[
U:=\PP^2\setminus\{p_1\}\cong\Tot_{L_1}\OO(1),\quad
q:U\longrightarrow L_1.
\]
Some power of the ideal of $L_1$ annihilates $F$, so $F$ is finite over $L_1$. 
Purity implies that 
$$E:=q_*F$$ is a vector bundle. 
At $p_2$ the normal coordinate $y=x_1/x_2$ has character $n:=u/v$.
Then 
\begin{equation}
P_2(F)=(1-n)[E|_{p_2}].
\label{eq:PW}
\end{equation}

\begin{lemma}
\label{lem:string}
If $\Norm{P_2(F)}=2$, there is a character $\rho$ such that
\begin{equation}
[E|_{p_2}]=\rho(1+n+\cdots+n^{d-1}),\quad P_2(F)=\rho(1-n^d).
\label{eq:string}
\end{equation}
\end{lemma}
\begin{proof}
We may write $[E|_{p_2}]$ as a sum over cosets of the group generated by $n=u/v$: 
$$[E|_{p_2}]=\sum_{\rho}\rho \sum_{j\in \mathbb{Z}}w_jn^j. $$
As
$$(1-n)\sum_{j}w_jn^j=\sum_{j}(w_j-w_{j-1})n^j, $$
therefore
$$\Norm{P_2(F)}=\sum_{\rho}\sum_{j}(w_j-w_{j-1})^2. $$
By assumption, we know that there is only one nonzero $\rho$ and its
$w_j=1$ when $a\leqslant j\leqslant b$ for some $a,b$. 
By absorbing $n^a$ into $\rho$, and using the fact that $[E|_{p_2}]$ has length $d$, we get \eqref{eq:string}.
\end{proof}

The subtorus $\mathbb{C}^*_{u}\subset T$ fixes $L_1$ pointwise. 
By \eqref{eq:string}, the $\mathbb{C}^*_{u}$-eigenbundles
of $E$ have rank one, with consecutive distinct weights, thus
\begin{equation}
E=\bigoplus_{i=0}^{d-1}E_i=\bigoplus_{i=0}^{d-1}\OO(a_i),\quad
\phi_i:\OO(a_i)\rightarrow\OO(a_{i+1}+1).
\label{eq:chain}
\end{equation}
Every adjacent arrow is nonzero. Otherwise the module splits into two
nonzero summands, one of which has slope at least that of $F$, contrary
to the stability.

Near $p_2$, the normal action has the form $nE_i\to E_{i+1}$. Its source and target fibers have the same
character by \eqref{eq:string}. Since the other local coordinate $x_0/x_2$ has character
$v^{-1}$, independent of $n$, the nonzero equivariant coefficient is
constant there. Thus each arrow $\phi_i$ is nonvanishing at $p_2$ and is globally
a scalar multiple of
\begin{equation}
x_2^{m_i},\quad m_i=a_{i+1}-a_i+1\geqslant 0,
\label{eq:arrows}
\end{equation}
where the scalar automorphisms of the line bundles remove all nonzero constants.

We define
\begin{equation}
\cA_d=\left\{a=(a_0,\ldots,a_{d-1})\in\ZZ^d \bigg|
\begin{array}{l}
\sum_{i=0}^{d-1} a_i=1-d,\quad a_{i+1}\geqslant a_i-1\ (0\leqslant  i<d-1),\\
\sum_{i=j}^{d-1}(a_i+1)\leqslant 0\quad(1\leqslant  j<d)
\end{array}\right\}.
\label{eq:Ad}
\end{equation}

\begin{proposition}
\label{prop:classification}
The sheaves supported on $L_1$ and satisfying \eqref{eq:survivors}
are parametrized by $\cA_d$, where the sequence $a\in \cA_d$ corresponds to the
bundle \eqref{eq:chain} with arrows \eqref{eq:arrows}. The set $\cA_d$
is finite, and each corresponding fixed component is an isolated
reduced point. The same statements hold on $L_2$ by a change of coordinates.
\end{proposition}
\begin{proof}
The condition $\chi(F)=1$ is $\sum_i(a_i+1)=1$. To test stability,
we work over function field $\kappa=\CC(L_1)$. Since every arrow $\phi_i$ is nonzero, $F$ corresponds to the 
module $\kappa[z]/(z^d)$, whose nonzero proper submodules are exactly
ideals $(z^j)$. 
A saturated subsheaf of $F$ is described as follows: 
consider the subbundle of $E$: 
\[
E_j\oplus\cdots\oplus E_{d-1},\quad E_i=\OO(a_i), 
\]
which is invariant under $\phi$, since the induced map
to the quotient vanishes generically and hence everywhere.  
This defines a subsheaf of $F$. 
By the stability of $F$, we have 
\[
\frac{\sum_{i=j}^{d-1}(a_i+1)}{d-j}<\frac1d\,,
\quad(1\leqslant  j<d).
\]
The numerator is an integer and $0<(d-j)/d<1$, so these are precisely
the last inequalities in \eqref{eq:Ad}.

Conversely, a sequence in $\cA_d$ defines a finite module on $U$ by
\eqref{eq:chain}, \eqref{eq:arrows}. Its underlying bundle is
torsion-free, so the corresponding sheaf on $\PP^2$ is pure. It is annihilated by the
$d$-th power of the normal coordinate and has proper support over $L_1$. The preceding argument proves
stability and $\chi(F)=1$. Normalizing the weights gives
\eqref{eq:string}, so the sheaf satisfies \eqref{eq:survivors}.

When $d=1$, we have $\cA_1=\{(0)\}$. When $d\geqslant 2$, by $\sum_{i=1}^{d-1}(a_i+1)\leqslant 0$ and $\sum_i a_i=1-d$, 
we have $a_0\geqslant 0$. Since $a_i\geqslant a_0-i$, the total sum gives
an upper bound for $a_0$. 
Writing $m_i=a_{i+1}-a_i+1$, we have
\begin{equation}
\sum_{i=0}^{d-2}(d-1-i)m_i
=\frac{(d-1)(d-2)}2-da_0.
\label{eq:weightedpartition}
\end{equation}
This proves finiteness.

Finally, $1-\mathrm{def}(F)+\mathrm{obs}(F)$ is constant on a torus fixed connected component.
By Lemmata~\ref{lem:boundary}, \ref{lem:support}, a component having an element satisfying the equality
\eqref{eq:survivors} must lie in the locus of that equality, which has just been
shown to be finite. The torus fixed locus of a smooth
variety is smooth. Such components are therefore isolated reduced
points. In particular $\mathrm{def}(F)=0$, and \eqref{eq:norm} gives $\mathrm{obs}(F)=0$.
\end{proof}

\subsection{Contribution to invariants from remaining fixed locus}
\label{sec:euler}
We consider sheaves $F$ satisfying \eqref{eq:survivors} and supported
on $L_1$. Recall that  
$$G=\mathbb{L}\iota_{L_0}^*F, \,\,\, \mathrm{hence} \,\,\,P_2(F)=[\mathbb{L}\iota_{p_2}^*G]. $$ 
Note that $p_1\notin \mathrm{supp}(F)$, therefore by Lemma \ref{lem:string} and equivariant Riemann-Roch: 
\begin{align}
\chi_T(G,G)=\frac{(1-n^d)(1-n^{-d})}{1-n}\notag=\sum_{j=0}^{d-1}n^j-\sum_{j=1}^{d}n^{-j}.
\label{eq:Gpairing}
\end{align}
Together with \eqref{QchiFFGG}, \eqref{eq:ext}, this yields
\begin{equation}
[T_FM_d]-[V_{d,F}]
=\sum_{j=1}^{d}n^{-j}-\sum_{j=1}^{d-1}n^j.
\label{eq:virtualdifference}
\end{equation}
Neither representation has zero weights by
Proposition~\ref{prop:classification}. Write $c_1^{T}(n)=s-t$, then 
\begin{equation}
\frac{e_T(T_FM_d)}{e_T(V_{d,F})}
=\frac{\prod_{j=1}^{d}[-j(s-t)]}
{\prod_{j=1}^{d-1}[j(s-t)]}
=(-1)^d d(s-t).
\label{eq:eulerratio}
\end{equation}
The support equation is $x_1^d$, so $h_T(F)=-ds$. The contribution 
of this point to the integral \eqref{eq:localization} is 
\begin{equation}
\frac{-ds}{(-1)^d d(s-t)}=(-1)^d\frac{s}{t-s}. \nonumber \end{equation}
By interchange $s$ and $t$, the contribution from $F$ that support on $L_2$ is $(-1)^d\frac{t}{s-t}$. 
Their sum is 
\begin{equation}
(-1)^d\frac{s}{t-s}+(-1)^d\frac{t}{s-t}=(-1)^{d+1}. 
\label{eq:pair}
\end{equation}

\begin{proposition}
\label{prop:localization}
For every $d\geqslant 1$, we have
\begin{equation}
\int_{M_d}h\cdot e(V_d)=(-1)^{d+1}|\cA_d|. \end{equation}
\end{proposition}
\begin{proof}
By Proposition \ref{eq:localization}, the integral reduces to a finite sum over the torus fixed locus. 
Proposition~\ref{prop:vanishing} eliminates all other fixed components and Proposition~\ref{prop:classification} identifies the remaining locus
with two copies of $\cA_d$. By the calculation in \eqref{eq:pair}, we obtain the claim. 
\end{proof}

\subsection{Computation of $|\cA_d|$}

To compute the cardinality of the finite set $\cA_d$, we need the notion of break divisors on a finite graph. Let $K_d$ be the graph with $d$ vertices 
and one edge between vertices $i$ and $j$ for all $1\leqslant i<j\leqslant d$, 
its genus is 
\[
g_d=|E(K_d)|-|V(K_d)|+1=\frac{(d-1)(d-2)}2, 
\]
where $E(K_d)$ and $V(K_d)$ denote edges and vertices of $K_d$ respectively.

We use the following characterization of its break divisors
\cite[\S 2.2, Lem.\,2.4]{KRT}. An integer vector
$b=(b_0,\ldots,b_{d-1})$ is a \textit{break divisor} if and only if
\begin{equation}
\begin{split}
b_i&\geqslant 0,\quad \sum_i b_i=g_d,\\
\sum_{i\in S}b_i&\geqslant g_{|S|} \,\,\, \text{for every nonempty }S\subseteq\{0,\ldots,d-1\}.
\end{split}
\label{eq:breakdef}
\end{equation}
The symmetric group $\mathfrak S_d$ acts on the set $\operatorname{Break}(K_d)$ of break divisors by permuting vertices.

\begin{proposition}
\label{prop:bijection}
The map
\begin{equation}
\cA_d\rightarrow\operatorname{Break}(K_d)/\mathfrak S_d,
\quad
(a_0,\ldots,a_{d-1})\mapsto
[(a_0,a_1+1,\ldots,a_{d-1}+d-1)]
\label{eq:bijection}
\end{equation}
is a bijection for every $d\geqslant 1$.
\end{proposition}
\begin{proof}
For $d=1$ both sets consist of the zero vector. Suppose $d\geqslant 2$ and put
\begin{equation}
b_i=a_i+i.
\label{eq:bchange}
\end{equation}
The adjacent inequalities in \eqref{eq:Ad} are equivalent to
$b_{i+1}\geqslant b_i$. Since $a_0\geqslant 0$, all $b_i$ are nonnegative, and
\begin{equation}
\sum_{i=0}^{d-1}b_i=(1-d)+\frac{d(d-1)}2=g_d.
\label{eq:bsum}
\end{equation}
Subtracting $\sum_{i=r}^{d-1}(a_i+1)\leqslant 0$ from $\sum_i(a_i+1)=1$ gives
$\sum_{i=0}^{r-1}(a_i+1)\geqslant 1$. Equivalently,
\begin{equation}
\sum_{i=0}^{r-1}b_i\geqslant 
1-r+\frac{r(r-1)}2=g_r \,\,\,(1\leqslant  r<d).
\label{eq:bprefix}
\end{equation}
For a nondecreasing vector, the minimum sum over subsets of cardinality
$r$ is the sum of the first $r$ entries. Thus \eqref{eq:bsum} and
\eqref{eq:bprefix} are precisely \eqref{eq:breakdef}.

Conversely, each $\mathfrak S_d$-orbit has a unique nondecreasing
representative. For such a representative put $a_i=b_i-i$. Reversing
the above equalities gives every condition in \eqref{eq:Ad}. This
constructs the inverse.
\end{proof}

We use the following specialization of the orbit-counting theorem of
Konvalinka--Reineke--Tewari.  

\begin{theorem}[{\cite[Theorem 3.7, $m=1$]{KRT}}]
\label{thm:KRT}
For $d\geqslant 2$, we have
\begin{equation}
\left|\operatorname{Break}(K_d)/\mathfrak S_d\right|
=\frac{1}{d^2}\sum_{e\mid d}(-1)^{d+e}\mu(d/e)
\binom{2e-1}{e}.
\label{eq:KRT}
\end{equation}
The same formula has value one for $d=1$.
\end{theorem}

\begin{proof}[Proof of Theorem~\ref{thm on g=0}]
By Proposition~\ref{prop:bijection}, Theorem~\ref{thm:KRT}, and
$\binom{2e-1}{e}=\frac12\binom{2e}{e}$, we have
\begin{align}
|\cA_d|
&=\frac{1}{2d^2}\sum_{e\mid d}
(-1)^{d+e}\mu(d/e)\binom{2e}{e}\notag\\
&=\frac{(-1)^d}{2d^2}\sum_{k\mid d}
\mu(k)(-1)^{d/k}\binom{2d/k}{d/k}.
\label{eq:Nclosed}
\end{align}
The second equality uses $e=d/k$.  
Applying Proposition~\ref{prop:localization} gives
$$\int_{M_d}h\cdot e(V_d)=-\frac{1}{2d^2}\sum_{k\mid d}
\mu(k)(-1)^{d/k}\binom{2d/k}{d/k}. $$
This proves the claimed equality.

As for the required correspondence with GW invariants, we put $B_m=(-1)^m\binom{2m}{m}$. Then
\begin{align*}
\sum_{k\mid d}\frac{\langle\tau_0(\pi^*H^2) \rangle_{d/k}^{\DT_4}}{k^2}
&=\sum_{k\mid d}\frac{1}{k^2}\frac{1}{2(d/k)^2}
\sum_{\ell\mid d/k}\mu(\ell)B_{d/(k\ell)}\notag\\
&=\frac{1}{2d^2}\sum_{k\ell\mid d}\mu(\ell)B_{d/(k\ell)}\notag\\
&=\frac{1}{2d^2}\sum_{r\mid d}B_{d/r}
\sum_{\ell\mid r}\mu(\ell)\notag\\
&=\frac{B_d}{2d^2}.
\label{eq:coverproof}
\end{align*}
Here the third equality groups terms by $r=k\ell$, and the fourth
uses $\sum_{\ell\mid r}\mu(\ell)=\delta_{r,1}$.
The last equality uses the GW formula of \cite[Prop.~2]{KP}: 
\begin{equation*}
\mathrm{GW}_{0,d}(\pi^*H^2)
=\frac{(-1)^d}{2d^2}\binom{2d}{d}.
\qedhere
\end{equation*}

\end{proof}

\section{Proof of Theorem \ref{thm on g=1}}\label{sect on proof of thm g=1}
Notations as used in the previous section. 
For $a=(a_0,\ldots,a_{d-1})\in\cA_d$, put
\begin{equation}
S_1(a)=\sum_{i=0}^{d-1}i(a_i+1),\quad
S_2(a)=\sum_{i=0}^{d-1}a_i(a_i+1).
\label{d:eq:statistics}
\end{equation}
Let $\nu_i=u^iv^{a_0-a_i-i}$ with $c^T_1(\nu_i)=is+(a_0-a_i-i)t$. The equivariant resolution of $F\in M_d^{T}$ is
\begin{equation}
0\longrightarrow\bigoplus_i\OO(a_i-1)\otimes \nu_i u
\longrightarrow\bigoplus_i\OO(a_i)\otimes \nu_i
\longrightarrow F\longrightarrow0.
\label{d:eq:resolution}
\end{equation}

\begin{lemma}
For $F\in M_d^{T}$ presented as above, we have 
\begin{equation}
c_1(\det \mathbb{R}\pi_{M*}\mathbb{F})|_F=S_1(a)s+\left(a_0-S_1(a)-\frac{S_2(a)}2\right)t.
\label{d:eq:determinant}
\end{equation}
An equivariant lift of \eqref{equ on tau1} has restriction
\begin{equation}
\tau_1(H)|_F
=d\left(\frac d2-S_1(a)\right)s
+d\left(\frac{3-d}{2}+S_1(a)+\frac{S_2(a)}2\right)t.
\label{d:eq:tauweight}
\end{equation}
\end{lemma}
\begin{proof}
With the torus action \eqref{equ on torus action}, for every $b\in \mathbb{Z}_{\geqslant 0}$, we have 
$$R\Gamma_T(\PP^2,\OO(b))=\sum_{i,j\geqslant 0,\, i+j\leqslant b}u^i v^j, $$
then 
\[
c_1\bigl(R\Gamma_T(\PP^2,\OO(b))\bigr)=\frac{b(b+1)(b+2)}{6}(s+t).
\]
If we twist $\OO(b)$ by a character $w$, we have 
$$c_1\bigl(R\Gamma_T(\PP^2,\OO(b)\otimes \mathbb{C}_{w})\bigr)=\frac{(b+1)(b+2)}{2}w+\frac{b(b+1)(b+2)}{6}(s+t). $$
By Serre duality, the same formula holds if $b\in \mathbb{Z}$. 
Applying this to the $i$-th pair of terms in
\eqref{d:eq:resolution} gives
\[
(a_i+1)c^T_1(\nu_i)+\frac{a_i(a_i+1)}2t.
\]
Summing over $i$ and using $\sum_i(a_i+1)=1$, we obtain
\eqref{d:eq:determinant}.

Choose $H$ with restrictions $0,s,t$ at the three
fixed points, then $\pi_{M*}H^2=1$ and $\pi_{M*}H^3=s+t$. Write
$$[F]=\sum_\alpha\epsilon_\alpha \mathbb{C}_{z_\alpha}\OO(b_\alpha H),$$ 
then 
\begin{equation}
\begin{split}
\pi_{M*}\left(H\ch_2(\mathbb{F})+\frac32H^2\ch_1(\mathbb{F})\right)\bigg|_F=
\sum_\alpha\epsilon_\alpha b_\alpha z_\alpha
+\frac{s+t}{2}\sum_\alpha\epsilon_\alpha b_\alpha^2+\frac32\left((s+t)\sum_\alpha\epsilon_\alpha b_\alpha
+\sum_\alpha\epsilon_\alpha z_\alpha\right).
\end{split}
\label{d:eq:unnormexpand}
\end{equation}
In the $i$-th pair of \eqref{d:eq:resolution}, we take
$(b,z)=(a_i,c^T_1(\nu_i))$ and $(a_i-1,c^T_1(\nu_i)+s)$ with signs $+$ and $-$. Taking a sum over $i$ and subtracting $d$ times \eqref{d:eq:determinant} gives \eqref{d:eq:tauweight}.
\end{proof}

\begin{proposition}
\label{d:prop:descendent}
With orientation \eqref{equ on Md vir class}, for any $d\geqslant 1$, the descendent invariant is
\begin{equation}
\langle\tau_1(\pi^*H) \rangle_{d}^{\DT_4}=(-1)^{d+1}\sum_{a\in\cA_d}
\left(d-\frac32-2S_1(a)-\frac{S_2(a)}2\right).
\label{d:eq:weighted}
\end{equation}
\end{proposition}
\begin{proof}
Write \eqref{d:eq:tauweight} as $Us+Vt$. Its restriction at the other
point is $Ut+Vs$. By \eqref{eq:eulerratio}, the contribution of the pair,
including the sign from orientation, is
\[
-\frac{Us+Vt}{(-1)^d d(s-t)}
-\frac{Ut+Vs}{(-1)^d d(t-s)}
=(-1)^{d+1}\frac{U-V}{d}.
\]
Since $(U-V)/d=d-3/2-2S_1(a)-S_2(a)/2$, summing over points in $\cA_d$, we obtain 
$$\langle\tau_1(\pi^*H) \rangle_{d}^{\DT_4}=(-1)^{d+1}\sum_{a\in\cA_d}
\left(d-\frac32-2S_1(a)-\frac{S_2(a)}2\right). $$
By elementary computations, we obtain the claimed equality in Theorem \ref{thm on g=1}. 
\end{proof}

\providecommand{\bysame}{\leavevmode\hbox to3em{\hrulefill}\thinspace}
\providecommand{\MR}{\relax\ifhmode\unskip\space\fi MR }
\providecommand{\MRhref}[2]{%
 \href{http://www.ams.org/mathscinet-getitem?mr=#1}{#2}}
\providecommand{\href}[2]{#2}

\end{document}